\documentclass[reqno]{amsart}
\usepackage{latexsym}
\usepackage{amsmath,amsthm}
\usepackage{amssymb}
\usepackage{graphicx}
\usepackage{hyperref}

\input diagxy
 \xyoption{curve}

\def\R{\mathbb{R}}

\newtheorem{theorem}{Theorem}[]
\newtheorem{lemma} {Lemma}
\newtheorem{proposition} {Proposition}

\newtheorem{corollary} {Corollary}

\theoremstyle{definition}
\newtheorem{definition}{Definition}
\newtheorem{rem}{Remark}

\begin{document}

\title{Sub-Finsler geometry and nonholonomic mechanics}
\author{Layth M. Alabdulsada}

\address{Dep. of Math, College of Science, University of Al Qadisiyah, Al-Qadisiyah 58001, Iraq}
\email{layth.muhsin@qu.edu.iq}


\bibliographystyle{alpha}

\begin{abstract}
We discuss a variational approach to the length functional and its relation to sub-Hamiltonian equations on sub-Finsler manifolds.
Then, we introduce the notion of the nonholonomic sub-Finslerian structure and prove that the distributions are geodesically invariant concerning the Barthel non-linear connection.
We provide necessary and sufficient conditions for the existence of the curves that are abnormal extremals; likewise, we provide necessary and sufficient conditions for normal extremals to be the motion of a free nonholonomic mechanical system, and vice versa.
 Moreover, we show that a coordinate-free approach for a free particle is a comparison between the solutions of the nonholonomic mechanical problem and the solutions of the Vakonomic dynamical problem for the nonholonomic sub-Finslerian structure.
In addition, we provide an example of the nonholonomic sub-Finslerian structure. Finally, we show that the sub-Laplacian measures the curvature
of the nonholonomic sub-Finslerian structure.
\end{abstract}
\maketitle

\section{Introduction}
Sub-Finsler geometry and nonholonomic mechanics have attracted much attention recently; they are rich subjects with many applications.

Sub-Finsler geometry is a natural generalization of sub-Riemannian geometry. The sub-Riemannian metric was initially referred to as the Carnot-Carathéodory metric. J. Mitchell, \cite{MI85}, investigated the Carnot-Carathéodory distance between two points by considering a smooth Riemannian $n$-manifold $(M, g)$ equipped with a $k$-rank distribution $\mathcal{D}$ of the tangent bundle $TM$. A decade later,  M. Gromov \cite{Gr96} provided a comprehensive study of the above concepts. V. N. Berestovskii \cite{BE88} identified the Carnot-Caratheodory Finsler metric version as the Finsler counterpart of this metric, now commonly known as the sub-Finsler metric. In this study, our definition of the sub-Finsler metric closely aligns with the definition presented in previous works \cite{L.K, JCG}. The motivation behind studying sub-Finsler geometry lies in its pervasive presence within various branches of pure mathematics, particularly in differential geometry and applied fields like geometric mechanics, control theory, and robotics. We refer the readers to \cite{ABB, L.K1, BBLS17, LoMa00}.

Nonholonomic mechanics is currently a very active area of the so-called geometric mechanics \cite{KF00}. Constraints on mechanical systems are typically classified into two categories: integrable and nonintegrable constraints.
 {\em Nonholonomic mechanics}: constraints that are not holonomic; these might be constraints that are expressed in terms of the velocity of the coordinates that cannot be derived from the constraints of the coordinates (thereby nonintegrable) or the constraints that are not given as an equation at all \cite{Lewis}.
Nonholonomic control systems exhibit unique characteristics, allowing control of underactuated systems due to constraint nonintegrability. These problems arise in physical contexts like wheel systems, cars, robotics, and manipulations, with more insights found in \cite{BCGM15, KF00}.

In \cite{La01}, B. Langerock considered a general notion of connections over a vector bundle map and applied it to the study of mechanical systems with linear nonholonomic constraints and a Lagrangian of kinetic energy type.
A. D. Lewis in \cite{Lewis}, investigated various consequences of a natural restriction of a given affine connection to distribution. The basic construction comes from the dynamics of a class of mechanical systems with nonholonomic constraints. In a previous paper in collaboration with L. Kozma \cite{L.K}, constructed a generalized non-linear connection for a sub-Finslerian manifold, called $\mathcal{L}$-connection by the Legendre transformation which characterizes normal extremals of a sub-Finsler structure as geodesics of this connection.  In this paper, \cite{L.K} and \cite{L.K1} play an important role in calculating our main results. These results are divided into two parts: sub-Hamiltonian systems and nonholonomic sub-Finslerian structures on the nonintegrable distributions.

The paper is organized according to the following:
In Section 2, we review some standard facts about sub-Finslerian settings. In Section 3, we define a sub-Finsler metric on $\mathcal{D}$ by using a sub-Hamiltonian function $\eta(x,p)$ and show the correspondence between the solutions of sub-Hamiltonian equations and the solution of a variational problem. Section 4 introduces the notion of nonholonomic sub-Finslerian structures and presents the main results, including conditions for the motion of a free mechanical system under linear nonholonomic constraints to be normal extremal with respect to the linked sub-Finslerian structure. Section 5 provides an example of the nonholonomic sub-Finslerian structure, and Section 6 discusses the curvature of the sub-Finslerian structure. We conclude that if the sub-Laplacian $\Delta_F$ is zero, then the sub-Finslerian structure is flat and locally isometric to a Riemannian manifold, while if $\Delta_F$ is nonzero, the sub-Finslerian structure is curved and the shortest paths between two points on the manifold are not necessarily straight lines.

\section{Preliminaries}
Let $M$ be an $n$-dimensional smooth ($C^{\infty}$) manifold, and let $T_xM$ represent its tangent space at a point $x \in M$. We denote the module of vector fields over $C^{\infty}(M)$ by $\mathfrak{X}(M)$, and the module of $1$-forms by $\mathfrak{X}^*(M)$.

Consider $\mathcal{D}$, a {\it regular distribution} on $M$, defined as a subbundle of the tangent bundle $TM$ with a constant rank of $k$. Locally, in coordinates, this distribution can be expressed as $\mathcal{D} = \textrm{span} \{X_{1}, \ldots, X_{k}\}$, where $X_i(x) \in \mathfrak{X}(M)$ are linearly independent vector fields.

A non-negative function $F: \mathcal{D} \rightarrow \R_+$ is called a {\em sub-Finsler metric} if it satisfies the following conditions:
\begin{itemize}
  \item [1.] \textbf{Smoothness}: $F$ is a smooth function over $\mathcal{D}\setminus{0}$;
  \item [2.] \textbf{Positive Homogeneity}: $F(\lambda v) = |\lambda| F(v)$ for all $\lambda \in \R$ and $v \in \mathcal{D} \setminus {0}$;
  \item [3.] \textbf{Positive Definiteness}: The Hessian matrix of $F^2$ is positive definite at every $v \in \mathcal{D}_x \setminus {0}$.
\end{itemize}

A differential manifold $M$ equipped with a sub-Finsler metric $F$ is recognized as a {\em sub-Finsler manifold}, denoted by $(M, \mathcal{D}, F)$.

An absolutely continuous curve, denoted as $\sigma:[0, 1] \rightarrow M$, is considered {\em horizontal} if its tangent vector field $\dot{\sigma}(t)$ lies within $\mathcal{D}_{\sigma(t)}$ for all $t \in [0, 1]$, whenever it is defined. This condition reflects the nonholonomic constraints imposed on the curve.

The length functional of such a horizontal curve $\sigma$ possesses a derivative for almost all $t \in [0, 1]$, with the components of the derivative, $\dot{\sigma}$, representing measurable curves. The {\em length} of $\sigma$ is usually defined as:
$$\ell(\sigma)=\int_0^1 F(\dot{\sigma}(t))dt.$$

This length structure gives rise to a {\em distance function}, denoted as $d: M \times M \rightarrow \mathbb{R}_+$, defined by:
$$ d(x_0, x_1)=\inf \ell(\sigma), \qquad x_0, x_1 \in M,$$
and the infimum is taken over all horizontal curves connecting $\sigma(0) = x_0$ to $\sigma(1) = x_1$. This distance metric captures the minimal length among all possible horizontal paths between two points on the manifold $M$.


A  {\em geodesic}, also known as a  {\em minimizing geodesic}, refers to a horizontal curve $\sigma:[0, 1] \rightarrow M$  that realizes the distance between two points, i.e., $\ell(\sigma) = d(\sigma (0), \sigma (1))$.

 Throughout this paper, it is consistently assumed that $\mathcal{D}$ is bracket-generating. A distribution $\mathcal{D}$, is characterized as {\em bracket-generating} if every local frame $X_i$ of $\mathcal{D}$, along with all successive Lie brackets involving these frames, collectively span the entire tangent bundle $TM$.  If $\mathcal{D}$ represents a bracket-generating distribution on a connected manifold $M$, it follows that any two points within $M$ can be joined by a horizontal curve.
This foundational concept was initially established by C. Carathéodory \cite{CA09} and later reaffirmed by W. L. Chow \cite{Chow39} and P. K. Rashevskii \cite{RA38}. However, for a comprehensive explanation of the bracket-generating concept, one can turn to R. Montgomery's book, \cite{Mo02}.

\section{Sub-Hamiltonian associated with sub-Finslerian manifolds}
\subsection{The Legendre transformation and Finsler dual of sub-Finsler metrics}
Let $\mathcal{D}^*$ be a rank-$s$ codistribution on a smooth manifold $M$, assigning to each point $x \in U \subset M$ a linear subspace $\mathcal{D}^*_x \subset T^*_xM$. This codistribution is a smooth subbundle, and spanned locally by $s$ pointwise linearly independent smooth differential 1-forms:
$$\mathcal{D}_x^* = \mathrm{span}\{\alpha_1(x), \ldots, \alpha_s(x)\},\ \text{ with}\ \alpha_i(x) \in \mathfrak{X}^*(M).$$
We define the annihilator of a distribution $\mathcal{D}$ on $M$ as $(\mathcal{D}^{\bot})^0$, a subbundle of $T^*M$ consisting of covectors that vanish on $\mathcal{D}$:
$$(\mathcal{D}^{\bot})^0 = \{\alpha \in T^*M: \alpha(v) = 0 \text{ for all } v \in \mathcal{D}\},$$
such that $\langle v, \alpha \rangle := \alpha(v)$.
Similarly, we define the annihilator of the orthogonal complement of $\mathcal{D}$, denoted by $\mathcal{D}^0$, as the subbundle of $T^*M$ consisting of covectors that vanish on $TM^{\bot}$.

Using these notions, we can define a sub-Finslerian function denoted by $F^* \in \mathcal{D}^* \sim T^*M\setminus\mathcal{D}^0$, where $F^*$ is a positive function. This function shares similar properties with $F$, but is based on $\mathcal{D}^*$ instead of $\mathcal{D}$.

In our previous work \cite{L.K1}, we established the relationship:
\begin{equation}\label{Fpv}
F^*(p)= F(v), \ \mathrm{where}\ p =\mathcal{L}_L(v), \quad\mbox{for every}\quad  p \in \mathcal{D}^*_x \quad\mbox{and}\quad v \in \mathcal{D}_x,
\end{equation}
such that  $\mathcal{L}_L$ is the Legendre transformation of the sub-Lagrangian function $L: \mathcal{D} \subset TM \to \mathbb{R}$, a diffeomorphism between $\mathcal{D}$ and $\mathcal{D}^*$.

In this context, to express $F^*$ in terms of $F$, we consider the Legendre transformation of $F$ with respect to the sub-Lagrangian function $L(v) = \frac{1}{2} F(v,v)$, where $F(v,v)$ is the square of the Finsler norm of $v$. The Legendre transformation $\mathcal{L}_L$ maps $v \in \mathcal{D}$ to $p = \frac{\partial L}{\partial v}(v)$.

Utilizing the definition of the Legendre transformation, we observe that
$$p = \frac{\partial L}{\partial v}(v) = \frac{\partial}{\partial v}\left(\frac{1}{2}F(v,v)\right) = F(v,\cdot),$$
where $F(v,\cdot)$ denotes the differential of $F$ with respect to its first argument evaluated at $v$. Note that $F(v,\cdot)$ is a linear function on $\mathcal{D}_x$.

Given a covector $p \in \mathcal{D}^*$, with $x$ the base point of $\mathcal{D}$, we can express the dual sub-Finsler metric $F^*$ in terms of $F$ as
     $$
     F^*(p)= \sup_{v \in \mathcal{D}_x} \left\{ \frac{\langle p, v \rangle}{F(v)} \right\},
     $$
     where $\langle p, v \rangle$ represents the dual pairing between the covector $p$ and the vector $v$.

\subsection{The Sub-Hamiltonian Function and Sub-Hamilton's Equations for Sub-Finsler Manifolds}
      The sub-Hamiltonian function associated with a sub-Finsler metric $F$ given by
      \begin{equation*}
      \eta := \frac{1}{2}(F^*)^2.
        \end{equation*}
Here, $F^*$ denotes the dual metric to $F$, defined by
\begin{equation}\label{F*}
F^*(p) = \sup_{v \in \mathcal{D}_x, F(v)=1} \langle p,v\rangle,
\end{equation}
where $p$ represents a momentum vector in $\mathcal{D}^*_x$ associated with the point $x$ in the manifold $M$, and  $\langle \cdot,\cdot \rangle$ denotes the inner product induced by a Riemannian metric $g$. The sub-Finslerian metric defined by (\ref{F*}) is known as the Legendre transform of $F$, i.e., satisfying the relationship in \eqref{Fpv}.
It is worth noting that the sub-Hamiltonian function associated with a Finsler metric is not unique, and different choices of Hamiltonians may lead to different dynamics for the associated geodesics.

The sub-Hamiltonian formalism is a method of constructing a sub-Finsler metric on a subbundle $\mathcal{D}$ by defining a sub-Hamiltonian function $\eta(x,p)$ on the subbundle $\mathcal{D}^*$, where $x$ denotes a point in $M$ and $p$ denotes a momentum vector in $\mathcal{D}^*$, as explained in the following remark:
 \begin{rem}
The sub-Finsler vector bundle, introduced in \cite{L.K1} and expanded upon in \cite{LA23}, plays a pivotal role in formulating sub-Hamiltonians in sub-Finsler geometry. Consider the covector subbundle $(\mathcal{D}^*, \tau, M)$ with projection $\tau : \mathcal{D}^* \to M$, forming a rank-$k$ subbundle in the cotangent bundle of $T^*M$. The pullback bundle $\tau^*(\tau) = (\mathcal{D}^* \times \mathcal{D}^*, \mathrm{pr}_1, \mathcal{D}^*)$ is obtained by pulling back $\tau$ through itself and is denoted as the sub-Finsler bundle over $\mathcal{D}^*_x$. This bundle allows the introduction of $k$ orthonormal covector fields $X_1, X_2, \dots, X_k$ with respect to the induced Riemannian metric $g$. The sub-Hamiltonian $\eta$ induces a metric $g$ on the sub-Finsler bundle. In terms of this metric, the sub-Hamiltonian function $\eta$ can be expressed as a function of components $p_i$. Specifically, $\eta(x, p) = \frac{1}{2} \sum_{i,j=1}^{n} {g}^{ij} p_i p_j$, where $g^{ij}$ is the inverse of the metric tensor $g_{ij}$ for the extended Finsler metric $\hat{F}$ on $TM$, kindly check Remark \ref{EXT}. This defines a sub-Finsler metric on a subbundle $\mathcal{D}$ of $TM$ that is determined by a distribution on $M$. 
\end{rem}

Now fixing a point $x \in M$, for any covector $p \in \mathcal{D}^*$, there exists a unique {\em sub-Hamiltonian vector field} on $\mathcal{D}^*$,  denoted by $\vec{H}$, described by
 \begin{equation} \label{SH}
   \vec{H}= \frac{\partial \eta}{\partial p_i} \frac{\partial}{\partial x^i} - \frac{\partial \eta}{\partial x^i} \frac{\partial}{\partial p_i}.
 \end{equation}
 where the partial derivatives are taken with respect to the local coordinates $(x^i, p_i)$ on $\mathcal{D}^* \subset T^*M$.

\begin{definition} \label{NO}
The sub-Hamiltonian equations on $\mathcal{D}^*$ are then given by
\begin{subequations}\label{HE}
\begin{align}
  \frac{\partial \eta}{\partial p_i}  & = g^{ij} p_j, \label{HE1} \\
  \frac{\partial \eta}{\partial x^i}  & = -\frac{1}{2} \frac{\partial g^{jk}}{\partial x^i} p_j p_k.  \label{HE2}
\end{align}
\end{subequations}
\end{definition}
 These equations express the fact that the sub-Hamiltonian vector field $\vec{H}$ preserves the sub-Finsler metric $F^*$ on $\mathcal{D}^*$. If the Hamiltonian is independent of the cotangent variables $p_i$, then the second equation above reduces to the Hamilton-Jacobi equation for the sub-Finsler manifold $(M, \mathcal{D}, F)$.
\begin{rem} \label{EXT}
We extended sub-Finsler metrics to full Finsler metrics using an orthogonal complement subbundle in \cite{L.K}. However, here are more details and evidence.

Given a subbundle $\mathcal{D}$ of the tangent bundle $TM$, its direct complement $\mathcal{D}^\perp$ is a subbundle of $TM$ such that $TM = \mathcal{D} \oplus \mathcal{D}^\perp$, and at every point $x \in M$, $\mathcal{D}_x \cap \mathcal{D}_x^\perp = {0}$ and $\mathcal{D}_x + \mathcal{D}_x^\perp = T_xM$.

One canonical way to obtain a direct complement to $\mathcal{D}$ is to use the notion of an orthogonal complement. Given a subbundle $\mathcal{D}$ of $TM$, we define the orthogonal complement bundle $\mathcal{D}^\perp$ as follows:
\begin{equation*}
\mathcal{D}^\perp_x = \{v \in T_xM :\langle v,w\rangle=0 \text{ for all } w \in \mathcal{D}_x\},
\end{equation*}
such that $v, w$ are orthogonal with respect to the inner product induced by the Riemannian metric.
It can be shown that $\mathcal{D}^\perp$ is a subbundle of $TM$ and satisfies the conditions for being a direct complement to $\mathcal{D}$. Moreover, it can be shown that any two direct complements to $\mathcal{D}$ are isomorphic bundles, so the orthogonal complement is unique up to bundle isomorphism.

Note that if $M$ is equipped with a sub-Finsler metric, then the metric induces a non-degenerate inner product on $\mathcal{D}$, so we can use this inner product to define the orthogonal complement. However, if $M$ is not equipped with a Riemannian metric, then the notion of an orthogonal complement may not be well-defined. So, to extend a given sub-Finsler metric $F$ on a subbundle $\mathcal{D}$ of $TM$ to a full Finsler metric on $TM$, one can use an orthogonal complement subbundle $\mathcal{D}^{\perp}$. This is a regular subbundle of $TM$ that is orthogonal to $\mathcal{D}$ with respect to the Riemannian metric $g_{ij}$. Locally, $\mathcal{D}^{\perp}$ can be written as:

\begin{equation}
\mathcal{D}^{\perp} = \text{span}\{ X'_1,\ldots,X'_{n-k}\},
\end{equation}

where $k$ is the rank of the subbundle $\mathcal{D}$ and $X'_1,\ldots,X'_{n-k}$ are local vector fields that form a basis for $\mathcal{D}^{\perp}$. Then, one can define a Finsler metric $\hat{F}$ on $TM$ by:

\begin{equation} \label{HF}
\hat{F}(v) = \sqrt{ F^2(P(v)) + \widetilde{F}^2(P^{c}(v)) },
\end{equation}

where $P$ is the projection onto $\mathcal{D}$, $P^{c}$ is the projection onto $\mathcal{D}^{\perp}$, and $\widetilde{F}$ is a Finsler metric on $\mathcal{D}^{\perp}$. This construction yields a full Finsler metric on $TM$ that extends the sub-Finsler metric $F$ on $\mathcal{D}$. Note that the Finsler metric $\widetilde{F}$ on $\mathcal{D}^{\perp}$ is not unique, so the choice of $\widetilde{F}$ is arbitrary. However, the resulting Finsler metric $\hat{F}$ on $TM$ is unique and independent of the choice of $\widetilde{F}$.

To see this, suppose we have two choices of Finsler metrics $\widetilde{F}$ and $\widetilde{F}'$ on $\mathcal{D}^{\perp}$. Let $\hat{F}$ and $\hat{F}'$ be the corresponding extensions of $F$ to $TM$ using Equation \ref{HF}. Then for any $v \in TM$, we have
\begin{align*}
\hat{F}^2(v) &= F^2(P(v)) + \widetilde{F}^2(P^c(v)) \\
\hat{F}'^2(v) &= F^2(P(v)) + \widetilde{F}'^2(P^c(v)).
\end{align*}
Subtracting these two equations, we obtain
\begin{equation*}
\hat{F}'^2(v) - \hat{F}^2(v) = \widetilde{F}'^2(P^c(v)) - \widetilde{F}^2(P^c(v)).
\end{equation*}
Since $v$ can be decomposed uniquely as $v = v_{\parallel} + v_{\perp}$ with $v_{\parallel} \in \mathcal{D}$ and $v_{\perp} \in \mathcal{D}^{\perp}$, we have $P^c(v) = v_{\perp}$, and the right-hand side of the above equation depends only on $v_{\perp}$. Since the choice of $\widetilde{F}$ on $\mathcal{D}^{\perp}$ is arbitrary, we can choose $\widetilde{F}$ and $\widetilde{F}'$ to be equal except on a single vector $v_{\perp}$, in which case $\widetilde{F}'^2(P^c(v)) - \widetilde{F}^2(P^c(v))$ will be nonzero only for that vector. Therefore, we have $\hat{F}'^2(v) - \hat{F}^2(v) \neq 0$ only for that vector, and hence $\hat{F} = \hat{F}'$.

Therefore, we have shown that the resulting Finsler metric $\hat{F}$ on $TM$ is unique and independent of the choice of $\widetilde{F}$.
\end{rem}

Let us turn to define the normal and abnormal extremals:

The projection $x(t)$ to $M$ is called a {\em normal extremal}.
 One can see that every sufficiently short subarc of the normal extremal $x(t)$ is a minimizer sub-Finslerian geodesic. This subarc is the unique minimizer joining its endpoints (see \cite{L.K1, L3}).
 In the sub-Finslerian manifold, not all the sub-Finslerian geodesics are normal (contrary to the Finsler manifold). This is because the sub-Finslerian geodesics, which admit a minimizing geodesic, might not solve the sub-Hamiltonian  equations. Those minimizers that are not normal extremals are called {\it singular} or {\it abnormal extremals}, (see for instance \cite{Mo02}).
 Even in the sub-Finslerian case, Pontryagin's maximum principle implies that every minimizer of the arc length of the horizontal curves is a normal or abnormal extremal.
 \subsection{Non-Linear Connections on a sub-Finsler manifolds}
 \begin{definition}
An $\mathcal{L}$-{\em connection} $\nabla$ on a sub-Finsler manifold is a {generalized non-linear connection} over the induced mapping
\begin{equation} \label{E-def}
E:T^*M \to TM, \ \ \ E(\alpha(x))= \mathbf{i}(\mathcal{L}_{\eta}(\mathbf{i}^*(\alpha(x))))\in TM,
\end{equation} constructed by Legendre transformation $\mathcal{L}_{\eta}: \mathcal{D}^* \subset T^*M \to \mathcal{D} \subset TM$ by (\ref{E-def}), where $\mathbf{i}^*:T^*M \to \mathcal{D}^*$ is the adjoint mapping of $\mathbf{i}:  \mathcal{D}  \rightarrow TM$, i.e. for any
$\alpha(x) \in \mathfrak{X}^*(M),$ $\mathbf{i}^*(\alpha(x))$ is determined by $$\langle X(x), \mathbf{i}^*(\alpha(x))\rangle = \langle \mathbf{i}(X(x)), \alpha(x) \rangle \ \text {for all} \ X(x) \in \mathfrak{X}(M),$$
such that $\langle v, \alpha \rangle := \alpha(v)$ for all $v \in \mathcal{D}, \alpha \in \mathcal{D}^*$. For more details about the settings of the $\mathcal{L}$- connection $\nabla$, we refer the reader to \cite{L.K}. Obviously, $E$ is a bundle mapping whose image set is precisely the subbundle $\mathcal{D}$ of $TM$ and whose kernel is the annihilator $\mathcal{D}^0$ of $\mathcal{D}$.
\end{definition}
Moreover, we recall the {\em Barthel non-linear connection} $\overline{\nabla}^B$ of the cotangent bundle as follows
$$\overline{\nabla}^B_X\alpha (Y) = X(\alpha(Y))- \alpha (\nabla^B_X Y),$$
where the Berwald connection $\nabla^B$ on the tangent bundle was locally given by
\begin{equation}
N_j^i= \frac{1}{2}\frac{\partial G^i}{\partial v^j};\quad
G^i = g^{ij}\left ( \frac{\partial^2 L}{\partial v^j \partial x^k}v^k - \frac{\partial L}{\partial x^j } \right). \label{barthel}
\end{equation}
The Barthel nonlinear connection plays the same role in the positivity homogeneous case as the Levi-Civita connection in Riemannian geometry, see \cite{K03}.
\begin{definition} \label{GEO}
A curve $\alpha:[0,1] \to T^*M$ is said to be $E$-{\it admissible} if $E(\alpha(t))=\dot{\sigma}(t)\ \forall t \in [0,1]$ such that $\pi_M:T^*M \to M$ is the natural cotangent bundle projection. An {\em auto-parallel} curve is the $E$-admissible curve with respect to $\mathcal{L}$-connection $\nabla$ if it satisfies $\nabla_{\alpha} \alpha(t) = 0$  for all $t \in [0,1]$. The {\em geodesic} of $\nabla$ is just the base curve $\gamma= \pi_M \circ \alpha$ of the auto-parallel curve.
\end{definition}
In coordinates, an auto-parallel curve $\alpha(t)=(x^i(t), p_i(t))$ satisfies the equations
$$\dot{x}^i(t) = g^{ij} (x(t), p(t)) p_j(t), \qquad  \dot{p}_i(t) = - \Gamma^{jk}_i(x(t), p(t)) p_j(t)p_k(t),$$
such that $g^{ij}$ and $\Gamma^{ik}_j $ are the local components of the contravariant tensor field
of $TM\otimes TM \to M$ associated with the sub-Hamiltonian structure and the connection coefficients of $\nabla$, respectively.
In fact, given a non-linear $\mathcal{L}$-connection $\nabla$ we can always introduce a smooth vector field
$\Gamma^{\nabla}$ on $\mathcal{D}^*$, in addition, their integral curves are auto-parallel curves in relation to $\nabla$.
In canonical coordinates, this vector field given by
$$\Gamma^{\nabla}(x, p)= g^{ij}(x, p)p_j \frac{\partial}{\partial x^i}-  \Gamma^{ik}_j(x, p) p_ip_k\frac{\partial}{\partial p_j}.$$
In \cite{L.K}, we proved that every geodesic of $\nabla$ is a normal extremal, and vice versa. More precisely, we have shown that the coordinate expression for the sub-Hamiltonian vector field (this is another form of (\ref{SH})) $\vec{H}$ equals:
$$\vec{H}(x, p) = g^{ij}(x, p)p_j  \frac{ \partial}{\partial x^i} - \frac{1}{2} \frac{\partial g^{ij}}{\partial x^k}(x, p) p_ip_j \frac{ \partial}{\partial p_k}.$$
Comparing the latter formula with the definition of $\Gamma^{\nabla}$,
yields that
 $\Gamma^{\nabla}(x, p)= \vec{H}(x, p)$.

\subsection{Variational approach to the length functional and its relation to sub-Hamiltonian equations on sub-Finsler manifolds}
 We can consider a small variation $\psi(s,t)$ of the curve $\sigma(t)$ such that $\psi(s,0)$ and $\psi(s,1)$ are fixed at $x_0$ and $x_1$, respectively, and $\psi(0,t) = \sigma(t)$ for all $t\in[0,1]$. We can think of $\psi(s,t)$ as a one-parameter family of curves in the set of all curves joining $x_0$ and $x_1$, and we can consider the variation vector field $v(t) = \frac{\partial \psi}{\partial s}(0,t)$, which is tangent to the curve $\sigma(t)$.

Then, we can define the directional derivative of the length functional $\ell$ along the variation vector field $v$ as
\begin{equation}\label{Var1}
\mathbf{d}\ell(\sigma)\cdot v = \frac{d}{ds}\Big|_{s=0} \ell(\psi(s,\cdot)).
\end{equation}
Note that $\ell(\psi(s,\cdot))$ is the length of the curve $\psi(s,\cdot)$, which starts at $x_0$ and ends at $x_1$. Therefore, $\frac{d}{ds}\Big|_{s=0} \ell(\psi(s,\cdot))$ is the rate of change of the length of the curve as we vary it along the vector field $v$.

By chain rule, we can write
$$\frac{d}{ds}\Big|_{s=0} \ell(\psi(s,\cdot)) = \int_0^1 \frac{\partial \ell}{\partial x^a}(\sigma(t))\frac{\partial \psi^a}{\partial s}(0,t)dt,$$
where $\frac{\partial \ell}{\partial x^a}$ is the gradient of the length functional. Using the fact that $\psi(s,t)$ is a variation of $\sigma(t)$ and $\psi(0,t) = \sigma(t)$, we can express $\frac{\partial \psi^a}{\partial s}(0,t)$ in terms of the variation vector field $v$ as
$$\frac{\partial \psi^a}{\partial s}(0,t) = \frac{\partial}{\partial s}\Big|_{s=0} \psi^a(s,t) = \frac{\partial}{\partial s}\Big|_{s=0} \sigma^a(t) + s\frac{\partial}{\partial t}\Big|{t=t} v^a(t) = v^a(t).$$
Therefore, we obtain
$$\frac{d}{ds}\Big|_{s=0} \ell(\psi(s,\cdot)) = \int_0^1 \frac{\partial \ell}{\partial x^a}(\sigma(t))v^a(t)dt = \mathbf{d}\ell(\sigma)\cdot v,$$
which gives the desired equation (\ref{Var1}).

Let us clarify the correct relationship between the sub-Hamiltonian equations and the length functional.

Given a sub-Finsler manifold $(M, \mathcal{D}, F)$, the sub-Hamiltonian equations on $M$ are given by
\begin{equation}\label{SHE}
  \frac{d}{dt}\left(\frac{\partial F}{\partial p_a}(\sigma(t))\right) = -\frac{\partial F}{\partial x^a}(\sigma(t)),
\end{equation}
where $\sigma: [0,1] \rightarrow M$ is a horizontal curve in $M$ with $\sigma(0) = x_0$ and $\sigma(1) = x_1$.

On the other hand, the length functional on $M$ is defined as
$$\ell(\sigma) = \int_0^1 F(\sigma(t), \dot{\sigma}(t)),dt,$$
where $\sigma$ is a horizontal curve in $M$ with $\sigma(0) = x_0$ and $\sigma(1) = x_1$.

It will be shown (see Proposition 6) that a curve $\sigma$ is a solution to the sub-Hamiltonian equations if and only if it is a critical point of the length functional $\ell$. In other words, if $\sigma$ satisfies the sub-Hamiltonian equations, then $\mathbf{d}\ell(\sigma) = 0$, and conversely, if $\sigma$ is a critical point of $\ell$, then it satisfies the sub-Hamiltonian equations.
\begin{proposition}\label{LFH}
A horizontal curve $\sigma:[0,1]\rightarrow M$ joining $\sigma(0)=x_0$ with $\sigma(1)=x_1$ in $M$ is a solution to the sub-Hamiltonian equations if and only if it is a critical point of the length functional $\ell$. That is, if and only if $\mathbf{d}\ell(\sigma) = 0$.
\end{proposition}
\begin{proof}
We will begin by proving the first direction:

Assume that $\sigma$ satisfies the sub-Hamiltonian equations. Then, we have
$$\frac{d}{dt}\left(\frac{\partial F}{\partial p_a}(\sigma(t))\right) = -\frac{\partial F}{\partial x^a}(\sigma(t)),$$
for all $a=1,\ldots,m$ and $t\in[0,1]$. Note that $\frac{\partial F}{\partial p_a}$ is the conjugate momentum of $x^a$, and we can write the sub-Finsler Lagrangian as
$$L(x,\dot{x}) = F(x,\dot{x})\sqrt{\det(g_{ij}(x))},$$
where $g_{ij}(x) = \frac{\partial^2 F^2}{\partial \dot{x}^i \partial \dot{x}^j}(x,\dot{x})$ is the sub-Finsler metric tensor. Then, the length functional can be written as
$$\ell(\sigma) = \int_0^1 L(\sigma(t),\dot{\sigma}(t)),dt.$$
Using the Euler-Lagrange equation for the Lagrangian $L$, we have
$$\frac{d}{dt}\left(\frac{\partial L}{\partial \dot{x}^a}(\sigma(t),\dot{\sigma}(t))\right) - \frac{\partial L}{\partial x^a}(\sigma(t),\dot{\sigma}(t)) = 0,$$
for all $a=1,\ldots,m$ and $t\in[0,1]$. Since $L$ depends only on $\dot{x}$ and not on $x$ explicitly, we can write this as
$$\frac{d}{dt}\left(\frac{\partial F}{\partial \dot{x}^a}(\sigma(t))\sqrt{\det(g_{ij}(\sigma(t)))}\right) - \frac{\partial F}{\partial x^a}(\sigma(t))\sqrt{\det(g_{ij}(\sigma(t)))} = 0,$$
for all $a=1,\ldots,m$ and $t\in[0,1]$. Using the chain rule and the fact that $\sigma$ is horizontal curve, we can write this as
$$\frac{d}{dt}\left(\frac{\partial F}{\partial p_a}(\sigma(t))\right) - \frac{\partial F}{\partial x^a}(\sigma(t)) = 0,$$
for all $a=1,\ldots,m$ and $t\in[0,1]$. This is exactly the condition for $\sigma$ to be a critical point of $\ell$, i.e., $\mathbf{d}\ell(\sigma) = 0$.

Now, let us proceed to prove the second direction:

Assume that $\sigma$ is a critical point of $\ell$, i.e., $\mathbf{d}\ell(\sigma) = 0$. Then, for any smooth variation $\delta\sigma:[0,1]\rightarrow TM$ with $\delta\sigma(0) = \delta\sigma(1) = 0$, we have
$$0 = \mathbf{d}\ell(\sigma)(\delta\sigma) = \int_0^1\left\langle\frac{\partial F}{\partial x^a}(\sigma(t)),\delta x^a(t)\right\rangle dt,$$
where $\delta x^a(t) = \frac{d}{ds}\bigg|_{s=0}x^a(\sigma(t)+s\delta\sigma(t))$ is the variation of the coordinates $x^a$ induced by $\delta\sigma$. Note that we have used the fact that $\delta\sigma(0) = \delta\sigma(1) = 0$ to get rid of boundary terms.

Since $\delta\sigma$ is arbitrary, this implies that
$$\frac{\partial F}{\partial x^a}(\sigma(t)) = 0,$$
for all $a=1,\ldots,m$ and $t\in[0,1]$. Using the sub-Hamiltonian equations, we can write this as
$$\frac{d}{dt}\left(\frac{\partial F}{\partial p_a}(\sigma(t))\right) = 0,$$
for all $a=1,\ldots,m$ and $t\in[0,1]$.
This implies that $\frac{\partial F}{\partial p_a}$ is constant along $\sigma$. Since $\sigma$ is horizontal curve, we can choose a partition $0=t_0<t_1<\cdots<t_n=1$ such that $\sigma$
is smooth on each subinterval $[t_{i-1},t_i]$. Let $c_a$ be the constant value of $\frac{\partial F}{\partial p_a}$ on $\sigma$.

Then, for each $i=1,\ldots,n$, we have
$$\frac{d}{dt}\left(\frac{\partial F}{\partial p_a}(\sigma(t))\right) = 0,$$
for all $a=1,\ldots,m$ and $t\in[t_{i-1},t_i]$. This implies that
$$\frac{\partial F}{\partial p_a}(\sigma(t)) = c_a,$$
for all $a=1,\ldots,m$ and $t\in[t_{i-1},t_i]$. Since $\frac{\partial F}{\partial p_a}$ is the conjugate momentum of $x^a$, this implies that $\sigma$ satisfies the sub-Hamiltonian equations on each subinterval $[t_{i-1},t_i]$.

Therefore, $\sigma$ satisfies the sub-Hamiltonian equations on the whole interval $[0,1]$, which completes the proof of second direction.
\end{proof}
 \begin{corollary}
  If $\sigma:[0,1]\rightarrow M$ is a horizontal curve that minimizes the length functional $\ell$ between two points $x_0$ and $x_1$ on a sub-Finsler manifold $(M, \mathcal{D}, F)$, then $\sigma$ is a smooth geodesic between $x_0$ and $x_1$. Conversely, if $\sigma$ is a smooth geodesic between $x_0$ and $x_1$, then its length $\ell(\sigma)$ is locally minimized.
 \end{corollary}
 \begin{proof}
   The proof of this corollary follows directly from Proposition \ref{LFH}.
 \end{proof}

The Proposition \ref{LFH} establish the significance of the results in the context of sub-Hamiltonian equations and curve optimization on a sub-Finsler manifold.
The corollary highlights the connection between curve optimization, geodesics, and the length functional on sub-Finsler manifolds. Collectively, these results provide deep insights into the geometric behavior of curves on sub-Finsler manifolds, linking the sub-Hamiltonian equations, length minimization, and the concept of geodesics in this context.


\section{Nonholonomic sub-Finslerian structure}
A sub-Finslerian structure is a generalization of a Finslerian structure,  where the metric on the tangent space at each point is only required to be positive-definite on a certain subbundle of tangent vectors.

A {\it nonholonomic sub-Finslerian structure} is a triple
$(M, \mathcal{D}, F)$ where $M$ is a smooth manifold of dimension $n$, $\mathcal{D}$
 is a non-integrable distribution of rank
$k <  n$ on $M$, which means that it cannot be generated by taking the Lie bracket of vector fields. This property leads to the nonholonomicity of the structure and has important implications for the geometry and dynamics of the system.  The regularity condition on $\mathcal{D}$ means that it can be locally generated by smooth vector fields, and the nonholonomic condition means that it cannot be integrable to a smooth submanifold of $M$.
The sub-Finslerian metric $F$ is a positive-definite inner product on the tangent space of $\mathcal{D}$ at each point of $M$. It is often expressed as a norm that satisfies the triangle inequality but does not necessarily have the homogeneity property of a norm. The metric $F$ induces a distance function on $M$, known as the sub-Riemannian distance or Carnot-Carathéodory distance, which is a natural generalization of the Riemannian distance.
 Mechanically, sub-Riemannian manifolds $(M, \mathcal{D}, g)$ and their generalization,  sub-Finslerian manifolds $(M, \mathcal{D}, F)$ are classified as configuration spaces \cite{L.K2}.

Nonholonomic sub-Finslerian structures arise in the study of control theory and robotics, where they model the motion of nonholonomic systems, i.e., systems that cannot achieve arbitrary infinitesimal motions despite being subject to arbitrary small forces. The motivation for this generalization comes from the need to provide a framework that captures the complexities of motion in such systems beyond what sub-Riemannian geometry alone can achieve. The study of these structures involves geometric methods, such as the theory of connections and curvature, and leads to interesting mathematical problems. This generalization not only extends the applicability of the theory to a wider class of problems but also paves the way for new insights into the geometric mechanics of nonholonomic systems.

\subsection{Nonholonomic Free Particle Motion under a Non-Linear Connection and Projection Operators}
We have the projection operator $P^* : T^*M \to \mathcal{D}^0$ that projects any covector $\alpha \in T^*M$ onto its horizontal component with respect to the non-linear connection induced by the distribution $\mathcal{D}$. More precisely, for any $Y \in TM$, we define $P(Y)$ to be the projection of $Y$ onto $\mathcal{D}$, and then $P^*(\alpha)(Y) = \alpha(Y - P(Y))$.

Next, we have the complement projection $(P^*)^c : T^*M \to (\mathcal{D}^\bot)^0$, which projects any covector $\alpha \in T^*M$ onto its vertical component with respect to the non-linear connection induced by the distribution $\mathcal{D}$. More precisely, for any $Y \in TM$, we define $P^\bot(Y)$ to be the projection of $Y$ onto $\mathcal{D}^\bot$, and then $(P^*)^c(\alpha)(Y) = \alpha(P^\bot(Y))$.

Now, we consider a nonholonomic free particle moving along a horizontal curve $\sigma : [0,1] \to M$. Let $\overline{\nabla}^B$ be a Barthel non-linear connection, (see \cite{L.K, La01}), and the condition $P^*(\overline{\nabla}^B_{\dot{\sigma}(t)} \dot{\sigma}(t)) = 0$ expresses the fact that the velocity vector $\dot{\sigma}(t)$ is constrained to be horizontal, while the constraint condition $\dot{\sigma}(t) \in \mathcal{D}^0$ expresses the fact that the velocity vector lies in the distribution $\mathcal{D}$.

Using the fact that $T^*M$ can be decomposed into its horizontal and vertical components with respect to the non-linear connection induced by the distribution $\mathcal{D}$, we can express any covector $\alpha \in T^*M$ as $\alpha = P^*(\alpha) + (P^*)^c(\alpha)$. Then, the constraint condition $\dot{\sigma}(t) \in \mathcal{D}^0$ can be written as $(P^*)^c(\mathrm{d}\sigma/\mathrm{d}t) = 0$.

Using the above decomposition of $\alpha$, we can rewrite the condition $P^*(\overline{\nabla}^B_{\dot{\sigma}(t)} \dot{\sigma}(t)) = 0$ as $P^*(\overline{\nabla}^B_{\dot{\sigma}(t)} \dot{\sigma}(t)) = P^*(\mathrm{d}\dot{\sigma}/\mathrm{d}t) = \mathrm{d}(P^*(\dot{\sigma}))/\mathrm{d}t = 0$, where we have used the fact that $P^*(\mathrm{d}\dot{\sigma}/\mathrm{d}t)$ is the derivative of the horizontal component of $\dot{\sigma}$ with respect to time, and hence is zero if $\dot{\sigma}$ is constrained to be horizontal.

Therefore, the conditions $P^*(\overline{\nabla}^B_{\dot{\sigma}(t)} \dot{\sigma}(t)) = 0$ and $(P^*)^c(\mathrm{d}\sigma/\mathrm{d}t) = 0$ together express the fact that the velocity vector $\dot{\sigma}(t)$ of the nonholonomic free particle is constrained to be horizontal and lie in the distribution $\mathcal{D}$, respectively.

Since $T^*M$ is identified with $TM$ via a Riemannian metric $g$, we have a natural isomorphism between $(\mathcal{D}^{\bot})^0$ and $\mathcal{D}^0$ given by the orthogonal projection. In particular, we have a direct sum decomposition of the cotangent bundle $T^*M$ as
$$T^*M \cong (\mathcal{D}^{\bot})^0 \oplus \mathcal{D}^0.$$
Note that any covector $\alpha \in T^*M$ can be uniquely decomposed as $\alpha = (P^*)^c(\alpha) + P^*(\alpha)$.

We can define a new {\it non-linear connection} $\overline{\nabla}$ on $(M, \mathcal{D}, F)$ according to
\begin{equation}\label{AAlA}
\overline{\nabla}_X (P^* (\alpha))(Y) =\overline{\nabla}^B_X (P^* (\alpha))(Y) +\overline{\nabla}^B_X ((P^*)^c (\alpha))(Y)
\end{equation} for all $X \in \mathfrak{X}(M)$ and $\alpha \in \mathfrak{X}^*(M)$. We restrict this connection to $\mathcal{D}^0$ and the equations of motion of the nonholonomic free particle can be re-written as $\overline{\nabla}_{\dot{\sigma}(t)} {\dot{\sigma}(t)}=0$, together with the initial velocity taken in $\mathcal{D}^0$ (see \cite{La01, Lewis}).

Given a nonholonomic sub-Finsler structure $(M, \mathcal{D}, F)$ one can always construct a normal and $\mathcal{D}$-adapted $\mathcal{L}$-connection \cite[Proposition 16]{L.K}.
Furthermore, we can construct a generalized non-linear connection over the vector bundle $\mathbf{i}: \mathcal{D} \to TM$, we will set $X \in \Gamma (\mathcal{D})$ with $\mathcal{L}_{\eta}(\mathbf{i} \circ X) \in \mathfrak{X}^*(M)$. So, attached to $(M, \mathcal{D}, F)$ there is a non-linear connection $\nabla^{H}: \Gamma (\mathcal{D}) \times \Gamma (\mathcal{D}^0) \to \Gamma (\mathcal{D}^0)$ called the {\it nonholonomic connection} over the adjoint mapping $\mathbf{i}:  \mathcal{D}  \rightarrow TM$ on natural projection $\tau: \mathcal{D}^0 \to M$ given by
$$\nabla^{H}_X \alpha(Y) = P^*(\overline{\nabla}^B_{X}\alpha(Y)).$$
Moreover, there is no doubt this indeed determines a non-linear connection, namely,
$$\nabla^{H}_{X}\alpha(Y)= \overline{\nabla}_X (P^* (\alpha))(Y),$$
such that $\overline{\nabla}$ is the non-linear connection given in (\ref{AAlA}), for all $X \in \Gamma(\mathcal{D})$ and $\alpha \in \Gamma(\mathcal{D}^0)$.
 In the nonholonomic setting, the horizontal curves are $\hat{\sigma}$ in $\mathcal{D}$ that are extensions of curves in $M$, i.e. $\hat{\sigma} (t)= \dot{\sigma}(t)$ for some curve $\sigma$ in $M$.
 \begin{definition} Let $(M, \mathcal{D}, F)$  be a nonholonomic sub-Finsler structure.
  A {\em nonholonomic bracket} $$[\cdot,\cdot] : \Gamma (\pi_{\mathcal{D}})\otimes \Gamma (\tau) \to \Gamma (\tau)$$ is defined as
  $[X, \alpha]= (P^*)^c[X, \alpha] $ for all $X \in  \Gamma (\pi_{\mathcal{D}})$, $\alpha \in  \Gamma (\tau)$,  $\pi_{\mathcal{D}} : \mathcal{D} \to M$ and $\tau: \mathcal{D}^* \to M$.
  This Lie bracket satisfies all the regular properties of the Lie bracket with the exception of the Jacobi identity.
   It may happen that the nonholonomic bracket $[X, \alpha]\notin \Gamma (\tau)$  because $\mathcal{D}^*$ is nonintegrable.
 \end{definition}
 Now, we can formally define the torsion operator
 $$T (X,\alpha):=\nabla^{H}_{X}\alpha -  \nabla^{H}_{\alpha}X- P^*[X, \alpha].$$
 In this setting, due to the symmetry of the non-linear connection $\nabla^{H}_{X}\alpha=\nabla^{H}_{\alpha}X$, the torsion $T (X, \alpha) = 0$  for all $X \in \Gamma(\mathcal{D})$ and $\alpha \in \Gamma(\mathcal{D}^0)$.
 Moreover,  \cite[Lemma 5]{L.K1}, implies that the non-linear connection $\nabla^{H}$ preserves the sub-Finsler metric $F$ on $\mathcal{D}$, i.e. $\nabla^{H}_{X}F=0$ for all $X \in \Gamma(\mathcal{D})$.  Therefore, there exists a unique conservative homogeneous nonlinear connection $\nabla^{H}$ with zero torsion
 and we can write the equations of motion for the given nonholonomic problem as $\nabla^{H}_{\dot{\sigma}(t)}\dot{\sigma}(t) = 0$, in such a way that $\sigma$ is a curve in $M$ tangent to $\mathcal{D}$.

There is a close relationship between nonholonomic constraints and the controllability of non-linear systems. More precisely,
there is a beautiful link between optimal control of nonholonomic systems and sub-Finsler geometry.
In the case of a large class of physically interesting systems, the optimal control problem is reduced to finding geodesics with respect to the sub-Finslerian metric.
The geometry of such geodesic flows is exceptionally rich and provides guidance for the design of control laws, for more information see Montgomery \cite{Mo02}.
We have seen in Section 2 that for each point $x \in M$,  we have the following distribution of rank $k$
$$\mathcal{D}= \textrm{span} \{X_{1}, \ldots, X_{k}\}, \qquad X_i(x) \in T_xM,$$
such that for any control function $u(t) = (u_1(t), \ldots, u_t(t))\in \mathbb{R}^k$ the control system is defined as
$$ \dot{x}=\sum_{i=1}^{k} u_iX_{i}(x), \qquad x \in M,$$
 is called a {\em nonholonomic control system} or {\em driftless control system} in the quantum mechanical sense, see \cite{L.K2}.

\subsection{Results}
The subsequent findings enhance comprehension of nonholonomic sub-Finslerian structures and their relevance in geometric mechanics. These insights offer essential tools for addressing and resolving issues concerning restricted movement within mathematical and physical domains. Specifically, these results shed light on the behavior of nonholonomic structures and their utility in analyzing constrained motion, particularly within the realm of geometric mechanics.
\begin{rem}\label{GI}
We call the distribution $\mathcal{D}$ a {\em geodesically invariant} if for every geodesic $\sigma: [0, 1] \to M$ of $\overline{\nabla}^B$, $\dot{\sigma}(0) \in \mathcal{D}_{{\sigma}(0)}$ implies that $\dot{\sigma}(t) \in \mathcal{D}_{{\sigma}(t)}$ for every $t \in (0,1]$.

  One can prove that if $(M, \mathcal{D}, F)$ is a sub-Finslerian manifold such that for any $x \in M$, $\mathcal{D}_x$ is a vector subspace of $T_xM$. The distribution $\mathcal{D}$ is geodesically invariant if and only if, for any $x \in M$ and any $v \in \mathcal{D}_x$, the Jacobi field along any geodesic $\gamma(t)$ with initial conditions $\gamma(0)=x$ and $\dot{\gamma}(0)=v$ is also in $\mathcal{D}$.

In other words, if the Jacobi fields along any geodesic with initial conditions in $\mathcal{D}$ remain in $\mathcal{D}$, then $\mathcal{D}$ is geodesically invariant. Conversely, if $\mathcal{D}$ is geodesically invariant, then any Jacobi field along a geodesic with initial conditions in $\mathcal{D}$ must also remain in $\mathcal{D}$. We leave the proof of this statement for future work.
\end{rem}
The following Proposition implies, in particular, that $\mathcal{D}$ is geodesically
invariant with respect to Barthel's non-linear connection $\overline{\nabla}^B$.
\begin{proposition} \label{12}
  \item
 \begin{itemize}
     \item [(I)] For each $X \in \mathfrak{X}(M)$ and $\alpha  \in \Gamma(\mathcal{D}^0)$, $\overline{\nabla}_X (P^*(\alpha))(Y)\in \Gamma(\mathcal{D}^0)$.
     \item [(II)] For each $X \in \mathfrak{X}(M)$ and $\alpha  \in \Gamma(\mathcal{D}^0)$, $\overline{\nabla}^B_X ((P^*)^c(\alpha))(Y) \in \Gamma(\mathcal{D}^0)$.
     \item [(III)] For each $X \in \mathfrak{X}(M)$ and $\alpha  \in \Gamma(\mathcal{D}^\bot)^0$, $\overline{\nabla}^B_X ((P^*)^c(\alpha))(Y)  \in \Gamma(\mathcal{D}^\bot)^0$.
   \end{itemize}
\end{proposition}
 \begin{proof}
 \item
\begin{itemize}
     \item [(I)] Let $X \in \mathfrak{X}(M)$ and $\alpha \in \Gamma(\mathcal{D}^0)$. Then, by the definition of the pullback connection, given in (\ref{AAlA}), and the Leibniz rule, we have
\begin{align*}
\overline{\nabla}_X(P^*(\alpha))(Y) &= X(P^*(\alpha)(Y)) - P^*(\alpha)(\overline{\nabla}_X(Y)) \\
&= X(\alpha(P(Y))) - \alpha(\overline{\nabla}_X(Y)) \\
&= \alpha(X(P(Y))) - \alpha(\overline{\nabla}_X(Y)) \\
&= \alpha(P(\mathcal{L}_X(Y))) - \alpha(\overline{\nabla}_X(Y)) \\
&= P(\alpha(\mathcal{L}_X(Y))) - \alpha(\overline{\nabla}_X(Y)) \\
&= P(\mathcal{L}_X(\alpha(Y))) - \alpha(\overline{\nabla}_X(Y)) \\
&= P(\mathcal{L}_X(P^*(\alpha)(Y))) - \alpha(\overline{\nabla}_X(Y)) \\
&= P(\overline{\nabla}_X(P^*(\alpha))(Y)) - \alpha(\overline{\nabla}_X(Y)).
\end{align*}
Since $P(\overline{\nabla}_X(P^*(\alpha))(Y))$ and $\alpha(\overline{\nabla}_X(Y))$ both lie in $\Gamma(\mathcal{D}^0)$, it follows that $\overline{\nabla}_X(P^*(\alpha))(Y)$ also lies in $\Gamma(\mathcal{D}^0)$.

\item [(II)]
 Using the definition of the connection $\overline{\nabla}^B$, we have:
\begin{align*}
    \overline{\nabla}^B_X ((P^*)^c(\alpha))(Y) &= X((P^*)^c(\alpha)(Y)) - (P^*)^c(\alpha)(\nabla_X^B Y) \\
    &\quad + (P^\bot)^c(\alpha)(\nabla_X^B Y).
\end{align*}

Now, let us analyze each term on the right-hand side individually:

First, consider $X((P^*)^c(\alpha)(Y))$. Since $(P^*)^c(\alpha)(Y)$ is a section of $\mathcal{D}^0$ and $X$ is a vector field on $M$, $X((P^*)^c(\alpha)(Y))$ is a section of $\mathcal{D}^0$.

Next, we have $-(P^*)^c(\alpha)(\nabla_X^B Y)$. Here, $(P^*)^c(\alpha)$ is a bundle map from $\mathcal{E}$ to $\mathcal{D}^0$, so $(P^*)^c(\alpha)(\nabla_X^B Y)$ is a section of $\mathcal{D}^0$. The negative sign in front ensures that the result remains in $\mathcal{D}^0$.

Finally, we consider $(P^\bot)^c(\alpha)(\nabla_X^B Y)$. Since $(P^\bot)^c(\alpha)$ is a bundle map from $\mathcal{E}$ to the orthogonal complement of $\mathcal{D}^0$, $(P^\bot)^c(\alpha)(\nabla_X^B Y)$ is a section of $\Gamma(\mathcal{D}^\bot)$. However, we need it to be a section of $\Gamma(\mathcal{D}^0)$.

To ensure that $(P^\bot)^c(\alpha)(\nabla_X^B Y)$ lies in $\Gamma(\mathcal{D)^0}$, we can use the projection operator $P$ to project it back onto $\mathcal{D}^0$. This projection ensures that the final result remains within $\Gamma(\mathcal{D}^0)$.

Combining these results, we see that $\overline{\nabla}^B_X ((P^*)^c(\alpha))(Y)$ is a section of $\Gamma(\mathcal{D}^0)$, as desired.
\item [(III)]
Using the definition of the connection $\overline{\nabla}^B$, we have
\begin{align*}
\overline{\nabla}^B_X ((P^*)^c(\alpha))(Y) &= X[(P^*)^c(\alpha)(Y)] - (P^*)^c(\alpha)(\overline{\nabla}_XY) + (P^*)^c(\overline{\nabla}^B_X\alpha)(Y)\\
&= X[(P^*)^c(\alpha)(Y)] - (P^*)^c(\alpha)(\overline{\nabla}_XY) + (P^*)^c((\overline{\nabla}_X\alpha)^\top)(Y)\\
&= X[(P^*)^c(\alpha)(Y)] - (P^*)^c(\alpha)(\nabla_X^B Y) + (P^*)^c((\overline{\nabla}_X\alpha)^\top)(Y)
\end{align*}
where in the last step we used the fact that $$(P^*)^c(\alpha)(\nabla_X^B Y) = -(P^*)^c(\alpha)(\overline{\nabla}_XY),$$ which follows from the definition of the codifferential operator and the fact that $(P^*)^c = -(P^*)^c$.

Now we need to show that the three terms on the right-hand side of this expression lie in $\Gamma(\mathcal{D}^\bot)^0$. We will do this term by term.
First, note that $(P^*)^c(\alpha)(Y) \in \Gamma(\mathcal{D}^\bot)^0$ since $(P^*)^c(\alpha)$ maps $\Gamma(\mathcal{D}^\bot)$ to itself and $Y \in \Gamma(\mathcal{D}^\bot)^0$.

Next, we need to show that $(P^*)^c(\alpha)(\nabla_X^B Y) \in \Gamma(\mathcal{D}^\bot)^0$. Note that $$(P^*)^c(\alpha)(\nabla_X^B Y) = - (P^*)^c(\alpha)(\overline{\nabla}_XY),$$ so it suffices to show that $(P^*)^c(\alpha)(\overline{\nabla}_XY) \in \Gamma(\mathcal{D}^\bot)^0$. To see this, note that $\overline{\nabla}_XY \in \Gamma(\mathcal{D}^\bot)^0$ since $X$ and $Y$ are both sections of $\mathcal{D}^\bot$, and that $(P^*)^c(\alpha)$ maps $\Gamma(\mathcal{D}^\bot)^0$ to itself.

Finally, we need to show that $(P^*)^c((\overline{\nabla}_X\alpha)^\top)(Y) \in \Gamma(\mathcal{D}^\bot)^0$. To see this, note that $(\overline{\nabla}_X\alpha)^\top$ is a tensor of type $(1,1)$ that maps vectors tangent to $M$ to vectors tangent to $M$, so $(P^*)^c((\overline{\nabla}_X\alpha)^\top)(Y)$ is a section of $\mathcal{D}^\bot$. Moreover, $(P^*)^c((\overline{\nabla}_X\alpha)^\top)$ maps $\Gamma(\mathcal{D}^\bot)^0$ to itself since $(\overline{\nabla}_X\alpha)^\top$ maps $\Gamma(TM)$ to itself and $(P^*)^c$ maps $\Gamma(\mathcal{D}^\bot)$ to itself.

Therefore, we have shown that $\overline{\nabla}_X\alpha \in \Gamma(\mathcal{D}^\bot)^0$, which implies that $\alpha$ is a harmonic one-form with respect to the induced metric on $\partial M$.

To summarize, we showed that if $\alpha$ is a closed one-form on $M$ such that $\alpha|_{\partial M}=0$, then $\alpha$ is a harmonic one-form with respect to the induced metric on $\partial M$.
\end{itemize}
\end{proof}

  In the following, we shall present the nonholonomic sub-Finslerian structure results. To begin, we define coordinate independent conditions for the motion of a free mechanical system subjected to linear nonholonomic constraints to be normal extremal with respect to the connected sub-Finslerian manifold, and vice versa.
Then, we address the problem of characterizing the normal and abnormal extremals that validate both nonholonomic and Vakonomic equations for a free particle subjected to certain kinematic constraints.

Let $(M, \mathcal{D}, F)$ be a nonholonomic sub-Finslerian structure and $\sigma:[0,1] \to M$ be a horizontal curve tangent to $\mathcal{D}$, then $\sigma$ is said to be a normal extremal if there exists $E$-admissible curve $\alpha$ with base curve $\sigma$ that is auto-parallel with respect to a normal $\mathcal{L}$-connection (Definition \ref{GEO}). While the curve $\sigma$  is said to be an abnormal extremal if there exists $\gamma \in \Gamma( \mathcal{D}^0)$ along $\sigma$  such that $\nabla_\alpha \gamma (t)= 0$ for all $t \in [0,1]$, with $\alpha$ a $E$-admissible curve with base curve $\sigma$.
\begin{rem}
  Cort\'{e}s et al. \cite{CLM}, made a comparison between the solutions of the nonholonomic mechanical problem and the solutions of the Vakonomic dynamical problem for the general Lagrangian system.
  The Vakonomic dynamical problem, associated with a free particle with linear nonholonomic constraints, consists of finding normal extremals with respect to the sub-Finsleriann structure $(M, \mathcal{D}, F)$.
  It is an interesting comparison because the equations of motion for the mechanical problem are derived by means of d'Alembert's principle, while the normal extremals are derived from a variation principle. Our next results are an alternate approach to the Cort\'{e}s results, that is a coordinate-free approach, for the free particle case in the sub-Finslerian settings.
\end{rem}

\begin{definition}
  Let $(M, \mathcal{D}, F)$ be a nonholonomic sub-Finslerian structure, one can establish new tensorial operators according to the following:
  \begin{align*}
    T^B&: \Gamma(\mathcal{D})\otimes \Gamma(\mathcal{D}^*) \to \Gamma(\mathcal{D}^0), \quad (X, \alpha) \mapsto P^*(\overline{\nabla}_X^B\alpha); \\
     T &: \Gamma(\mathcal{D})\otimes \Gamma(\mathcal{D}^0) \to \Gamma(\mathcal{D}^{\bot})^0, \quad (X, \gamma) \mapsto (P^*)^c(\delta_X\gamma);
  \end{align*}
   such that $$\delta: \Gamma(\mathcal{D}) \times \Gamma(\mathcal{D}^0)  \to \mathfrak{X}^*(M), \quad (X, \gamma)\mapsto \delta_X \gamma=i_Xd\gamma.$$
\end{definition}
In addition, these tensorial operators have the following properties:
  \begin{itemize}
    \item [(I)]  $T^B$ and $T$ are $\mathcal{F}(M)$-bilinear in their independent variables;
    \item [(II)] The behavior of $T^B$ and $T$ can be identified pointwise;
    \item [(III)] $T_x^B(X, \alpha)$ and $T_x(X, \gamma)$ have a clear and unequivocal meaning for all $X \in \mathcal{D}, \alpha \in \mathcal{D}^*$ and $\gamma \in \mathcal{D}^0$.
  \end{itemize}
In the following, w show the relation between the operator $T^B$ and the curvature of the distribution $\mathcal{D}$ using the following condition:

Suppose $X \in \mathcal{D}, \alpha \in \mathcal{D}^*$, then one has
$$\langle T(X, \gamma), \alpha \rangle = \langle \delta_X\gamma, \alpha \rangle= - \langle \gamma, [X,\alpha] \rangle,$$
for any $\gamma \in \Gamma(\mathcal{D}^0).$
Therefore, $T$ is trivial if and only if $\mathcal{D}$ is involutive.

\begin{definition}
Let $\nabla^T$ denote the non-linear connection over $i : \mathcal{D} \to TM$ on $\mathcal{D}^0$ by the following formula
  $$\nabla^T_X \gamma = P^*(\delta_X\gamma),$$
  such that $X \in \Gamma(\mathcal{D})$ and $\gamma \in \Gamma(\mathcal{D}^0).$
\end{definition}
\begin{proposition}
   Let $(M, \mathcal{D}, F)$ be a nonholonomic sub-Finslerian structure, assume that $\sigma: [0,1] \to M$ is a horizontal curve on $\mathcal{D}$ and let
   $\nabla$ be a $\mathcal{D}$-adapted $\mathcal{L}$-connection. Then, the following properties are satisfied:
   \begin{itemize}
   \item [(I)] If $p_0 \in \mathcal{D}^*_{\sigma(0)}$ is a given initial point, then $p(t)=\tilde{p}(t)$ for each $t \in [0,1]$ if and only if $T^B(\dot{\sigma}(t), \tilde{p}(t))= 0$, such that $p(t)$ and $\tilde{p}(t)$ are parallel transported curves along $\sigma$ w.r.t. $\overline{\nabla}^B$ and $\nabla^{H}$, respectively.
\item [(II)] If $\gamma_0 \in \mathcal{D}^0_{\sigma(0)}$ is a given initial point, then $\gamma(t) = \tilde{\gamma}(t)$ for each $t \in [0,1]$  if and only if $T(\dot{\sigma}(t), \tilde{\gamma}(t))= 0$, such that $\gamma(t)$ and $\tilde{\gamma}(t)$ are parallel transported curves along $\sigma$ w.r.t. $\nabla$  and $\nabla^{T}$, respectively.
   \end{itemize}
\end{proposition}
\begin{proof}
  It is sufficient to prove that the first case and the second one follow similar arguments.

  As a consequence of the definition of the tensorial operator $T^B$,
  for any section $S(t)$ of $\mathcal{D}^*$ along $\sigma$, the next expression is true
  $$\nabla^{H}_{{\dot{\sigma}(t)}}S(t)= \overline{\nabla}^B_{{\dot{\sigma}(t)}}S(t)- T^B({\dot{\sigma}(t)}, S(t)).$$
  Now, suppose that $S(t)= \tilde{p}(t)= p(t)$, then we get,
  $$T^B({\dot{\sigma}(t)}, S(t))= 0.$$
Conversely, it is well known that, regarding any connection, the parallel transported curves are uniquely determined by their initial conditions.
\end{proof}
It is clear that the second property of the above Proposition yields necessary and sufficient conditions for the existence of the curves that have abnormal extremals. In other words,
$\sigma$ is an abnormal extremal if and only if there exists a parallel transported section $\tilde{\gamma}$ of $\mathcal{D}^0$ along $\sigma$ with respect to $\nabla^{T}$ such that $T(\dot{\sigma}(t), \tilde{\gamma}(t))= 0$.
Now, by the next Proposition, one can derive the necessary and sufficient condition for normal extremals to be a motion of a free nonholonomic mechanical system and vice versa.
\begin{lemma}
  Let $(M, \mathcal{D}, F)$ be nonholonomic sub-Finslerian structures, and
   $\nabla$ be a normal non-linear $\mathcal{L}$-connection. Then for any $\alpha \in \mathfrak{X}^*(M)$ we have that $\nabla_{\alpha}\alpha= 0$ if and only if
   \begin{align*}
     \nabla^H_{E(\alpha)} \alpha(P)= & - T(E(\alpha),(P^*)^c(\alpha)); \\
     \nabla^T_{E(\alpha)}P^*(\alpha)= &   - T^B(E(\alpha),  \alpha(P)).
   \end{align*}
   \end{lemma}

   \begin{proof}
     We proved in \cite{L.K}, that  $\nabla_{\alpha}\alpha= 0$ if and only if $\nabla_{\alpha} \alpha = \overline{\nabla}^B_{E(\alpha)} (P^*)^c(\alpha)+ \delta_{E(\alpha)} P^*(\alpha)=0.$
     Moreover, $P^*(\alpha) = \alpha (P)$ and the Barthel non-linear connection preserves the metric, i.e. $\nabla^B \circ \mathcal{L}_{\eta} = \mathcal{L}_{\eta} \circ \overline{\nabla}^B$, therefore
     \begin{align*}
        \overline{\nabla}^B_{E(\alpha)}P^*(\alpha)= & \nabla^{H}_{E(\alpha)}P^*(\alpha)+T^B(E(\alpha),P^*(\alpha)), \\
       \delta_{E(\alpha)}P^*(\alpha)= & \nabla^{T}_{E(\alpha)}P^*(\alpha)+T(E(\alpha),(P^*)^c(\alpha)).
     \end{align*}
     According to the fact that $T^*M$ can be written as the direct sum of $(\mathcal{D}^{\bot})^0$ and $\mathcal{D}^0$,
     so the equivalence is pretty clear.
   \end{proof}
   \begin{theorem}
     If $\sigma :[0, 1]\rightarrow M$ is a solution of a free nonholonomic system given by nonholonomic sub-Finslerian structures, then it is also
  a solution of the corresponding Vakonomic problem, and vice versa, if and only if there exists $\gamma \in \Gamma(\mathcal{D}^0)$ along
  $\sigma$ such that
  \begin{equation}\label{Last}
    \nabla^T_{\dot{\sigma}}  \gamma (t)= - T^B(\dot{\sigma}(t), \mathcal{L}_{L}(\dot{\sigma}(t))),
  \end{equation}  further, for all $t$,
  $\gamma (t)\in {(\mathcal{D}_{\sigma(t)}+ [\dot{\sigma}(t), \mathcal{D}_{\sigma(t)}])}^0$.
   \end{theorem}
\begin{proof}  $\nabla_{\alpha} \alpha(t) = 0$ is the condition for any $E$-admissible curve $\alpha(t) = \mathcal{L}_{L}(\dot{\sigma}(t))+\gamma (t)$
to be parallel transported with respect to a normal $\mathcal{L}$-connection. In other words,
  $$\nabla^H_{\dot{\sigma}} \mathcal{L}_{L}(\dot{\sigma}(t))= - T(\dot{\sigma}(t), \gamma (t))$$ and
  $$\nabla^T_{\dot{\sigma}}\gamma (t)=- T^B(\dot{\sigma}(t),  \mathcal{L}_{L}(\dot{\sigma}(t))).$$
  Therefore, $\nabla^H_{\dot{\sigma}} \mathcal{L}_{L}(\dot{\sigma}(t))=0$ if and only if $T(\dot{\sigma}(t), \gamma (t))=0$, such that
  $\gamma (t)$ is a solution of (\ref{Last}).
  Since Remark \ref{GI} and Proposition \ref{12} guaranteed that $\mathcal{D}$ is geodesically invariant, therefore,
  given any $\gamma (t)$ in ${(\mathcal{D}_{\sigma(t)}+ [\dot{\sigma}(t), \mathcal{D}_{\sigma(t)}])}^0$, then
  (\ref{Last}) ensure that there is always a solution for all $t \in [0,1]$ not only for $\gamma (0)$ in ${(\mathcal{D}_{\sigma(0)}+ [\dot{\sigma}(0), \mathcal{D}_{\sigma(0)}])}^0$.
\end{proof}

\section{Examples from Robotics}
Typically, nonholonomic systems occur when velocity restrictions are applied, such as the constraint that bodies move on a surface without slipping.  Bicycles, cars, unicycles, and anything with rolling wheels are all examples of nonholonomic sub-Finslerian structures.

We will discuss the simplest wheeled mobile robot, which is a single upright rolling wheel, or unicycle, which is known as a kinematic penny rolling on a plane.
   Assume this wheel is of radius 1 and does not allow sideways sliding.
   Its configuration $M$ consists of the heading angle $\phi$, the wheel's point or the contact position $(x_1, x_2)$,
   and the rolling angle $\psi$ (see Figure \ref{KK}). Consequently, the space concerned has dimensions four, i.e., $M= \mathbb{R}^2 \times S^1 \times S^1$.
   There are two control functions deriving the wheel \cite{JCG, KF00}:
\begin{figure}[h]
  \centering
  \includegraphics[scale=0.45]{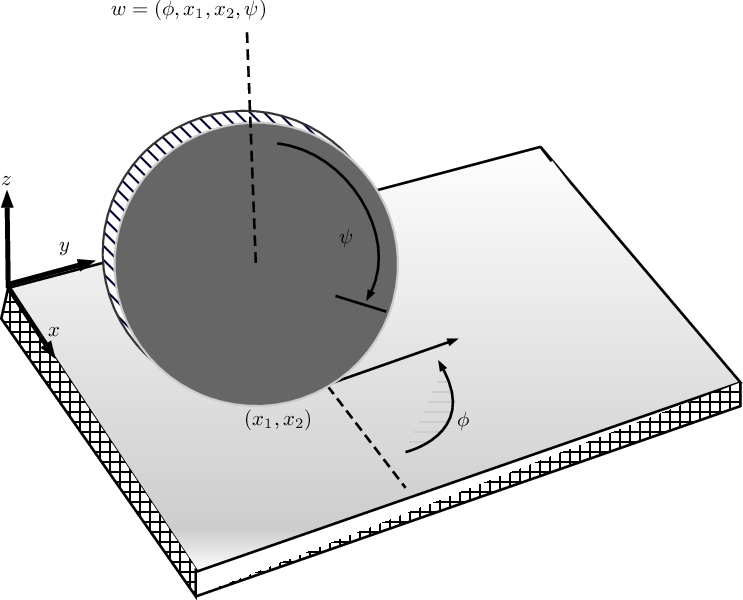}
  \caption{A kinematic penny rolling on a plane}\label{KK}
\end{figure}

   \begin{itemize}
     \item [(I)] $u_1$ [rolling speed], the forward-backward rolling angular,
          \item[(II)] $u_2$ [turning speed], the speed of turning the heading direction $\phi$.
   \end{itemize}
   With these controls, the rate of change of the coordinates can be expressed as follows:
   \begin{align}
    \dot{M} &= \begin{bmatrix}
    \dot{\phi}\\
           \dot{x}_{1} \\
           \dot{x}_{2} \\
           \dot{\psi}
         \end{bmatrix} =
         \begin{bmatrix}
         0 & 1 \\
            \cos {\phi} & 0\\
            \sin {\phi} & 0\\
           1 & 0\\
           \end{bmatrix}
         \begin{bmatrix}
           u_1 \\
           u_2
           \end{bmatrix} = X(M) u.
  \end{align}
  As we generally do not worry about the wheel's rolling angle, we could drop the fourth row from the above equation to get a simpler control system
\begin{align}
    \dot{M} &= \begin{bmatrix}
    \dot{\phi}\\
           \dot{x}_{1} \\
           \dot{x}_{2}
         \end{bmatrix} =
         \begin{bmatrix}
         0 & 1 \\
            \cos {\phi} & 0\\
            \sin {\phi} & 0
           \end{bmatrix}
         \begin{bmatrix}
           u_1 \\
           u_2
           \end{bmatrix} = X(M) u,
  \end{align}
which can be written as the following equation:
  $$X(M) u = X_1(M) u_1 + X_2(M) u_2,$$
  such that $u_1, u_2$ are called the controls and $X_1(w), X_2(w)$ are called vector fields.
  Moreover, each vector field assigns a velocity to every point $w$ in the configuration space,
  so these vector fields are sometimes called velocity vector fields. Hence the velocity vector fields
   of any solution curve should lie in $\mathcal{D}$ spanned by the following vector fields:
      \begin{align*}
  X_1(M) & = \cos {\phi} \frac{\partial}{\partial x_1} + \sin {\phi} \frac{\partial}{\partial x_2} + \frac{\partial}{\partial \psi}\\
 X_2(M)& = \frac{\partial}{\partial \phi}.
\end{align*}

 In a natural way, a sub-Riemannian metric on $\mathcal{D}$ is gained by asserting the vector fields $X_1(M), X_2(M)$ to be orthonormal vectors,
$$\langle u_1 X_1(M) + u_2 X_2(M),u_1 X_1(M) + u_2X_2(M)\rangle = u_1^2 + u_2^2.$$
The integral of this quadratic form measures the work completed in rolling the heading angle $\phi$ at the rate $\dot{\phi}$ and propelling
the wheel ahead at the rate of $\dot{\psi}$.
The sub-Riemannian structure will be adjusted as specified by the notion that curvature is costly:
namely, it takes more attempts to steer the wheel in a tight circle with little forward or backward movement than to steer it in a wide arc. Therefore, the curvature of the projection
$\sigma$ given by $\kappa = \frac{\dot{\phi}}{\dot{\psi}}$ this brings us to assume sub-Finsler metrics of the body
$$F = f (\kappa)\sqrt{d{\psi}^2 + d{\phi}^2},$$
such that $f$ grows larger but remains constrained as $\mid \kappa \mid$ increases.
After we check the sub-Finslerian property, one finds the nonholonomic of the rolling wheel, often known as a unicycle, by the equation
$\dot{M}=X(M)u$ which is the kinematic model of the unicycle.
\section{The sub-Laplacian associated with nonholonomic sub-Finslerian structures}
{\it The sub-Laplacian} is a differential operator that arises naturally in the study of nonholonomic sub-Finslerian structures. These are geometric structures that generalize Riemannian manifolds, allowing for non-integrable distributions of tangent spaces.

On a sub-Finslerian manifold $M$, there is a distinguished distribution of tangent spaces $\mathcal{D}$, which corresponds to the directions that are accessible by moving along curves with bounded sub-Finsler length. The sub-Finsler metric $F$ on $M$ measures the sub-Finsler length of curves with respect to this distribution.

The sub-Laplacian is defined as a second-order differential operator that acts on functions on $M$ and is defined in terms of the metric $F$ and the distribution $\mathcal{D}$. It is given by
$$\Delta_F = \mathrm{div}_{\mathcal{D}} (\mathrm{grad}_F),$$
where $\mathrm{grad}_F$ is the gradient vector field associated with $F$ which is the unique vector field satisfying $\mathrm{d}F(\mathrm{grad}_F, X) = X(F)$ for all vector fields $X$ on $M$, and $\mathrm{div}_{\mathcal{D}}$ is the divergence operator with respect to the distribution $\mathcal{D}$, which is defined as the trace of the tangential part of the connection on $\mathcal{D}$.

Our goal in this section is to show that the sub-Laplacian measures the curvature of the sub-Finslerian structure. It captures the interplay between the sub-Finsler metric $F$ and the distribution $\mathcal{D}$, and plays a crucial role in many geometric and analytic problems on nonholonomic sub-Finslerian manifolds.

For example, the heat kernel associated with the sub-Laplacian provides a way to study the long-term behavior of solutions to the heat equation on sub-Finslerian manifolds. The Hodge theory on sub-Finslerian manifolds is also intimately related to the sub-Laplacian, and involves the study of differential forms that are harmonic with respect to the sub-Laplacian.
\begin{rem}
To see that the sub-Laplacian measures the curvature of the sub-Finslerian structure, let us first recall some basic facts about Riemannian manifolds, see \cite{GL}. On a Riemannian manifold $(M,g)$, the Laplace-Beltrami operator is defined as
$$\Delta_g = \mathrm{div} (\mathrm{grad}_g),$$
where $\mathrm{grad}_g$ is the gradient vector field associated with the Riemannian metric $g$, and $\mathrm{div}$ is the divergence operator. It is a well-known fact that the Laplace-Beltrami operator measures the curvature of the Riemannian structure in the sense that it is zero if and only if the Riemannian manifold is flat.

The sub-Finslerian case is more complicated due to the presence of the distribution $\mathcal{D}$ that is not integrable in general. However, the sub-Laplacian $\Delta_F$ can still be understood as a curvature operator. To see this, we need to introduce the notion of a horizontal vector field.

A vector field $X$ on $M$ is called horizontal if it is tangent to the distribution $\mathcal{D}$. Equivalently, $X$ is horizontal if it is locally of the form $X = \sum_{i=1}^k h_i X_i$, where $h_i$ are smooth functions and $X_1,\ldots,X_k$ are smooth vector fields that form a basis for $\mathcal{D}$.

Given a horizontal vector field $X$, we can define its sub-Finsler length $|X|_F$ as the infimum of the lengths of horizontal curves that are tangent to $X$ at each point. Equivalently, $|X|_F$ is the supremum of the scalar products $g(X,Y)$ over all horizontal vector fields $Y$ with $|Y|_F \leq 1$.
\end{rem}

With these definitions in place, we can now show that the sub-Laplacian measures the curvature of the sub-Finslerian structure. More precisely, we have the following result:
\begin{theorem}\label{SLSF}
 The sub-Laplacian $\Delta_F$ is zero if and only if the sub-Finslerian manifold $(M,F,\mathcal{D})$ is locally isometric to a Riemannian manifold.
\end{theorem}
\begin{proof}
First, suppose that $(M,F,\mathcal{D})$ is locally isometric to a Riemannian manifold $(M,g)$. Then we can choose a local frame of orthonormal horizontal vector fields $X_1,\ldots,X_k$ with respect to the Riemannian metric $g$. In this frame, we have
$$\mathrm{grad}_F h = \sum_{i=1}^k g(\mathrm{grad}_h,X_i) X_i$$ for any function $h$ on $M$, and hence
$$\Delta_Fh= \sum_{i=1}^{k} \mathrm{div}_{\mathcal{D}}(g( \mathrm{grad}_h, X_i)X_i).$$
Using the fact that the $X_i$ form a basis for $\mathcal{D}$, we can rewrite this as
$$\Delta_Fh=\mathrm{div}( \mathrm{grad}_h)=\Delta_gh,$$
where $\Delta_g$ is the Laplace-Beltrami operator associated with the Riemannian metric $g$. Since $\Delta_g$ is zero if and only if $(M,g)$ is flat, it follows that $\Delta_F$ is zero if and only if $(M,F,\mathcal{D})$ is locally isometric to a Riemannian manifold, which implies that the sub-Finslerian structure is also flat.

Conversely, suppose that $\Delta_F$ is zero. Let $X_1,\ldots,X_k$ be a local frame of horizontal vector fields such that $F(X_i) = 1$ for all $i$, and let $\omega_{ij} = g(X_i,X_j)$ be the Riemannian metric induced by $F$ on $\mathcal{D}$. Using the definition of the sub-Laplacian and the fact that $\Delta_F$ is zero, we have
$$0= \Delta_FF=\mathrm{div}_{\mathcal{D}}( \mathrm{grad}_FF)=\sum_{i=1}^{k} \sum_{i=1}^{k} \frac{\partial^2F}{\partial x_i \partial x_j}\omega_{ij},$$
where $x_1,\ldots,x_k$ are local coordinates on $M$ that are adapted to $\mathcal{D}$ (i.e., $X_1,\ldots,X_k$ form a basis for the tangent space at each point). This implies that the Hessian of $F$ with respect to the Riemannian metric $\omega$ is zero, so $F$ is locally affine with respect to $\omega$. In other words, $(M,F,\mathcal{D})$ is locally isometric to a Riemannian manifold.
\end{proof}
\begin{rem}
In the above Theorem \ref{SLSF}, we have shown that the sub-Laplacian $\Delta_F$ measures the curvature of the sub-Finslerian structure. If $\Delta_F$ is zero, then the sub-Finslerian manifold is locally isometric to a Riemannian manifold, and hence the sub-Finslerian structure is flat. If $\Delta_F$ is nonzero, then the sub-Finslerian manifold is not locally isometric to a Riemannian manifold, and the sub-Finslerian structure is curved. This means that the shortest paths between two points on the manifold are not necessarily straight lines, and the geometry of the manifold is more complex than that of a Riemannian manifold.
\end{rem}

\paragraph*{\textbf{Acknowledgements}} The author gratefully acknowledges the helpful suggestions given by Dr. L\'aszl\'o Kozma during the preparation of the paper as well as his support for hosting me as a visiting research scholar at the University of Debrecen.

\end{document}